\documentclass[a4paper,11pt]{amsart}
\usepackage[utf8]{inputenc}
\usepackage{amsmath}
\usepackage{amssymb}
\usepackage{amsthm}
\usepackage{enumerate}
\usepackage{xcolor}
\usepackage[pagebackref=true, colorlinks]{hyperref}
\usepackage[margin=2.5cm]{geometry}
\usepackage{graphicx}

\date{\today}

\DeclareMathOperator{\id}{id}
\DeclareMathOperator{\Stab}{Stab}
\DeclareMathOperator{\Aut}{Aut}

\DeclareMathOperator{\Sym}{Sym}
\DeclareMathOperator{\Homeo}{Homeo}
\DeclareMathOperator{\Fix}{Fix}

\newcommand{\Full}{\mathbf{F}}
\newcommand{\defbold}{\textbf}

\newcommand{\grp}[1]{\langle #1 \rangle}

\newcommand{\Z}{\mathbb{Z}}
\newcommand{\G}{\mathcal{G}}

\newcommand{\N}{\mathbb{N}}

\newcommand{\B}{\mathcal{B}}

\newcommand{\atom}{\mathrm{at}}

\newcommand{\mc}[1]{\mathcal{#1}}
\newcommand{\rest}{\upharpoonright}

\newtheorem{theorem}{Theorem}[section]
\newtheorem{proposition}[theorem]{Proposition}
\newtheorem{corollary}[theorem]{Corollary}

\newtheorem{lemma}[theorem]{Lemma}

\theoremstyle{definition}
\newtheorem{qu}[theorem]{Question}
\newtheorem{definition}[theorem]{Definition}
\newtheorem{example}[theorem]{Example}
\newtheorem*{qu_intro}{Question}

\title{Discrete locally finite full groups of Cantor set homeomorphisms}

\author[A.\ Garrido]{Alejandra Garrido}
\address{Departamento de Matem\'aticas\\
Universidad Aut\'onoma de Madrid\\
 and ICMAT\\
 Madrid\\ 
 Spain}

\email{alejandra.garrido@uam.es\\ alejandra.garrido@icmat.es}

\author[C.\ D.\ Reid]{Colin D.\ Reid}
\address{School of Mathematical and Physical Sciences\\
	University of Newcastle\\
	University Drive\\
	2308\\
	Callaghan, Australia}
\email{colin@reidit.net}

\date{\today}

\thanks{This research was carried out while both authors were based at the University of Newcastle, Australia. The first-named author was, and the second-named author remains, a Postdoctoral Research Associate funded through ARC project \emph{Zero-dimensional symmetry and its ramifications} (FL170100032). The authors are grateful for the excellent working atmosphere and research group made possible by this funding.}


\begin{document}

\begin{abstract}
This work is motivated by the problem of finding locally compact group topologies for piecewise full groups (a.k.a.~ topological full groups).
We determine that any piecewise full group that is locally compact in the compact-open topology on the group of self-homeomorphisms of the Cantor set must be uniformly discrete, in a precise sense that we introduce here.
Uniformly discrete groups of self-homeomorphisms of the Cantor set are in particular countable, locally finite, residually finite and discrete in the compact-open topology.
The resulting piecewise full groups form a subclass of the ample groups introduced by Krieger.
We determine the structure of these groups by means of their Bratteli diagrams and associated dimension ranges ($K_0$ groups).
We show through an example that not all uniformly discrete piecewise full groups are subgroups of the ``obvious'' ones, namely, piecewise full groups of finite groups.
\end{abstract}

\maketitle

\section{Introduction}
During the last few years, group theorists have become increasingly interested in groups of homeomorphisms of the Cantor set, as sources of new finitely generated infinite simple groups (e.g., \cite{juschenkomonod,matui,nek_torsion_simple}). 
The groups in question consist of homeomorphisms of the Cantor set that can be pieced together from finitely many partial homeomorphisms between clopen subsets. 
For this reason, we call them \emph{piecewise full groups}.  We avoid the more common term ``topological full group'' as it could be confusing in the context of discussing topological groups, where the word ``topological'' plays a very different role.

In the developing theory of totally disconnected locally compact groups there is also reason to search for and study simple groups (\cite{capracemonod,CRW2}), especially those that are compactly generated and non-discrete, and so piecewise full groups are a natural place to search. 
Indeed, some of the well-known examples of this kind, Neretin's groups of almost automorphisms of locally finite trees,  can be naturally expressed as piecewise full groups (\cite{Kapoudjian, neretin}).

In searching for totally disconnected locally compact groups among piecewise full groups the first question that one faces is, which topology should the group be given? 
Since the group $\Homeo(X)$ of all self-homeomorphisms of the Cantor set $X$ is a well-known topological group (the typical topology is the compact-open topology, see Section \ref{sec:prelims}), one's first thought might be to study piecewise full groups with the induced topology from $\Homeo(X)$. 
Those familiar with Neretin's group know (\cite{newdirs_neretin}) that this group is not locally compact for the compact-open topology in $\Homeo(X)$. 
Indeed, we show here (Corollary~\ref{cor:full_lc_implies_ample}) that any piecewise full group that is locally compact for the compact-open topology of $\Homeo(X)$ must be \emph{uniformly discrete} -- there is a clopen partition of $X$ each of whose parts is visited at most once by each group orbit.
 These groups are a special type of what Krieger called an \emph{ample group} in \cite{krieger}: that is, they are countable piecewise full subgroups of $\Homeo(X)$, such that the group is locally finite (that is, every finite subset is contained in a finite subgroup) and such that the fixed-point set of every element is clopen.
We give in Proposition \ref{prop:udfull_direct_union_symmetrics} a description of these groups as direct limits of finite direct sums of symmetric groups.
Such groups have been studied before, notably in \cite{DudkoMedynets}, \cite{HartleyZalesskii} and \cite{lavnek} (in the latter two the summands are alternating groups).

There are examples showing that not all ample groups are uniformly discrete.
Indeed, ample groups have been previously studied as sources of simple groups, whereas (see Proposition \ref{prop:finite_orbit_implies_resinf}) all uniformly discrete ample groups are residually finite (the intersection of all finite-index subgroups is trivial).
This shows how far from our initial objective the compact-open topology brings us. 
Our second main result, Theorem \ref{thm:bratteli_condition},  is a characterisation of the uniformly discrete groups among these ample groups. 
This is done via Bratteli diagrams, which is a well-known and clear  way to encode direct limits of structures that decompose as direct sums. 
Since Bratteli diagrams provide a convenient dictionary between ample groups and their associated dimension ranges (or $K_0$ groups with unit), we also translate in Corollary \ref{cor:dimension_range} the Bratteli diagram condition to one on the dimension range. 
The fact that uniformly discrete ample groups are residually finite can also be gleaned from the structure of their Bratteli diagrams (Proposition ~\ref{prop:ud_ample_resfin_diagramproof}).


Finally, we consider how uniformly discrete ample groups arise.
Obvious examples are piecewise full groups of finite groups, and we are able  in Propositions \ref{prop:F_is_ud} and \ref{prop:full_finite_iff_top_bratteli} to distinguish these  among all uniformly discrete ample groups both from the topological dynamics on the Cantor set and from their associated Bratteli diagrams. 
Less obvious examples of uniformly discrete ample groups are afforded by stabilisers of piecewise full groups of finite groups. 
Example \ref{eg:not_stab_finite_ud} is a uniformly discrete ample group that cannot arise in this way. 
We leave open the following question:

\begin{qu_intro}
	Which uniformly discrete ample groups arise as subgroups of piecewise full groups of finite groups? 
	Can they be distinguished by their Bratteli diagrams?
\end{qu_intro}

As noted above, uniformly discrete ample groups are residually finite, which is a group-theoretic property. 
However, uniform discreteness is not preserved by group isomorphism. 
Section \ref{sec:example} is devoted to a straightforward example of a group with two different actions on the Cantor set, one uniformly discrete, the other not.

The original question of which topology to put on piecewise full groups to make them locally compact and totally disconnected is addressed in the forthcoming paper \cite{pseudogroups}. 
Showing the necessity of the approach taken there was in fact the motivation for the present paper.

\section{First properties of uniformly discrete groups}\label{sec:prelims}

\subsection{Notation}

Throughout, $X$ denotes the Cantor set. 
The group $\Homeo(X)$ of self-homeomorphisms of $X$ is a Polish (separable and completely metrisable) group when endowed with the \defbold{compact-open topology}, whose sub-basic open sets have the form
$$\{f\in\Homeo(X)\mid f(K)\subseteq O\}$$
where $K\subseteq X$ is compact and $O\subseteq X$ is open. 
This topology is the coarsest one on $\Homeo(X)$ that makes the action of $\Homeo(X)$ on $X$ jointly continuous; that is, such that the map $\Homeo(X)\times X \rightarrow X, (h, x)\mapsto h(x)$ is continuous for the product topology on $\Homeo(X)\times X$ (see \cite[p.224]{Kelley_topbook}).
Because $X$ is compact and metrisable, the compact-open topology is equivalent to the also commonly-used \defbold{topology of uniform convergence}, whose basic open sets have the form
$$\mc{N}(f, \epsilon)=\{g\in\Homeo(X)\mid \sup_{x\in X}d(f(x),g(x))<\epsilon \}  $$
where $f\in\Homeo(X)$, $\epsilon>0$ and $d$ is any (fixed) compatible  metric on $X$ (see  \cite[p. 230]{Kelley_topbook} or \cite[\S 46]{munkres_topbook}). 

The set of clopen subsets of $X$ ordered by inclusion forms a Boolean algebra, which is in fact  the unique (up to isomorphism) countably infinite, atomless Boolean algebra, denoted $\B$ here. 
The set of ultrafilters of $\B$ (equivalently, homomorphisms to the two-element Boolean algebra) is a compact, perfect and totally disconnected space (and so homeomorphic to $X$) when topologised so that the set of ultrafilters containing $b\in\B$ is a basic open, for each $b\in\B$. 
This is simply Stone's representation theorem, which says that, starting from $X$ and performing these two operations yields a naturally homeomorphic copy of $X$ (see \cite[IV.4]{univ_alg_book}). 
In particular, $\Homeo(X)$ is isomorphic to $\Aut(\B)$. 
It is not too hard to show (see, e.g., \cite[Theorem 2.3(d)]{BDK_cantortops}) that the compact-open topology on $\Homeo(X)$ 
 (equivalently, the topology of uniform convergence) then corresponds to the \defbold{permutation topology} or \defbold{topology of pointwise convergence} on $\Aut(\B)$, whose basic identity neighbourhoods are subgroups of the form
$$\Stab(\mc{F})=\{g\in\Aut(\B)\mid g(b)=b \text{ for all } b\in \mc{F}\} \text{ where } \mc{F}\subset\B \text{ is finite}.$$

In what follows we will identify $\Homeo(X)$ and $\Aut(\B)$ as groups, and use whichever of these two equivalent points of view seems most convenient.  Given a finite subalgebra $\mathcal{A}$ of $\mathcal{B}$, we write $\atom(\mathcal{A})$ for the set of \defbold{atoms} of $\mathcal{A}$, that is, the minimal nonzero elements.  Note that every finite Boolean algebra is generated by its atoms.

\begin{definition}\label{def:pwfull}
Given a group $G \leq \Homeo(X)$ and a $G$-invariant subalgebra $\mc{A}$ of $\mc{B}$, the \defbold{piecewise full group $\Full_{\mc{A}}(G)$ of $G$ with respect to $\mc{A}$} consists of all homeomorphisms $g\in\Homeo(X)$ for which there is a finite clopen partition $X=U_1\sqcup  \dots \sqcup U_n$ of $X$ with $U_1,\dots, U_n \in \mc{A}$ and $g_1,\ldots,g_n\in G$ such that $g\rest_{U_i}=g_i\rest_{U_i}$ for $i=1,\ldots,n$.  If $\mc{A} = \mc{B}$ we simply write $\Full(G) := \Full_{\mc{A}}(G)$ and call $\Full(G)$ the \defbold{piecewise full group} of $G$ (as a subgroup of $\Homeo(X)$).

A group $G\leq\Homeo(X)$ is called \defbold{piecewise full} if $\Full(G)=G$.
Note that for any subgroup $G$ of $\Homeo(X)$, we have $\Full(\Full(G)) = \Full(G)$, so $\Full(G)$ is piecewise full.

\end{definition}

The piecewise full group is known elsewhere in the literature as the \defbold{topological full group}.  We avoid this name as we will discuss topological groups, and the two different uses of ``topological'' here could be confusing.

\subsection{Uniformly discrete groups}

Our motivating problem is finding appropriate topologies to impose on a piecewise full subgroup of $\Homeo(X)$ that make it locally compact and non-discrete.  
The most obvious choice is the subspace topology in $\Homeo(X)$, and we have already seen that the compact-open topology on the latter is the coarsest one that makes the action on $X$ continuous. 
We shall presently see (Corollary \ref{cor:full_lc_implies_ample}) that this choice of topology forces piecewise full groups to be discrete in a strong sense, which we call uniformly discrete.
\begin{definition}\label{def:ud}
	A group $G\leq\Homeo(X)$ is \defbold{uniformly discrete} if there exists a clopen partition $\mc{P}$ of $X$ such that $|Gx\cap U|\leq 1$ for every $x\in X$ and $U\in \mc{P}$. 
\end{definition}
 
Suppose $G$ is a subgroup of $\Homeo(X)$  equipped with the compact-open topology.
Given a finite set $\mc{F}$ of clopen subsets of $X$, write
$$G_{(\mc{F})}:=G\cap \Stab(\mc{F})=\{g\in G \mid g(C)=C \text{ for all } C\in \mc{F}\}.$$
Note that as $\mc{F}$ ranges over the clopen partitions of $X$, the subgroups $G_{(\mc{F})}$ form a base of clopen neighbourhoods of the identity in $G$. 


\begin{lemma}\label{lem:basic_observations}
Let $G \le \Homeo(X)$ be equipped with the compact-open topology and let $\mc{P}$ be a clopen partition of $X$.
\begin{enumerate}[(i)]
\item $G$ is uniformly discrete with respect to $\mc{P}$ if and only if $\Full(G)$ is.
\item If $G$ is uniformly discrete with respect to $\mc{P}$ then $G$ is discrete: indeed, $G_{(\mc{P})}$ is trivial.
\end{enumerate}
\end{lemma}

\begin{proof}
Part (i) follows immediately from the observation that $G$ and $\Full(G)$ have the same orbits on $X$.
For part (ii), we see that if $G$ is uniformly discrete with respect to $\mc{P}$, then all orbits of $G_{(\mc{P})}$ must be singletons, ensuring $G_{(\mc{P})} = \{1\}$.
\end{proof}

In general, discrete subgroups of $\Homeo(X)$, even finite subgroups, need not be uniformly discrete (see Example~\ref{ex:c2_not_ud}).  However, the properties are equivalent in the context of piecewise full groups, and indeed they are the only way $G_{(\mc{P})}$ can be compact.

\begin{theorem}\label{thm:full_lc_implies_ample}
Let $G\leq \Homeo(X)$ be a piecewise full group equipped with the compact-open topology and let $\mc{P}$ be a clopen partition of $X$.  Then the following are equivalent:
\begin{enumerate}[(i)]
\item $G_{(\mc{P})}$ is compact;
\item $G$ is uniformly discrete with respect to $\mc{P}$;
\item $G_{(\mc{P})}$ is trivial.
\end{enumerate}
\end{theorem}

\begin{proof}
First observe that a compact subgroup of $\Homeo(X)$ must have finite orbits on clopen subsets of $X$. 
To wit,  the stabiliser of a clopen subset is open and, since it and its cosets form an open cover of the group, there can only be finitely many cosets; the orbit-stabiliser theorem then yields the claim. 

	We prove (i) $\Rightarrow$ (ii) via the contrapositive.  Suppose that $G$ is not uniformly discrete with respect to $\mc{P}$; that is, there is $x \in U \in \mc{P}$ and $g \in G$ such that $g(x) \in U \setminus \{x\}$.
	Since $g$ is a homeomorphism and $X$ is totally disconnected, there is a neighbourhood $U_x$ of $x$ such that $U_x\cap g(U_x)=\emptyset$ and $U_x, g(U_x)\subseteq U$. 
	Pick $x_1\in U_x\setminus\{x\}$ and a clopen neighbourhood $U_{x_1}\ni x_1$ such that $x \not\in U_{x_1}$. 
	Since $G$ is piecewise full it contains an element $g_1$ such that 
	\[
	g_1=\begin{cases} g &\text { on } U_{x_1}\\
	g^{-1} & \text{ on } gU_{x_1}\\
	\id &\text{ elsewhere }
	\end{cases}.
	\]
	Note that $g_1(U_{x})=g(U_{x_1})\sqcup (U_x \setminus U_{x_1})$ and $g_1\in G_{(\mc{P})}$. 
	Repeat the argument inductively to find a sequence $(g_n)_n$ of elements of $G_{(\mc{P})}$ supported on a decreasing sequence $(U_{x_n})_n$ of clopen subsets of $U_x$ and such that $g_n(U_x)=g(U_{x_n})\sqcup (U_x\setminus U_{x_n})$. 
	This produces infinitely many clopen subsets of $X$ in the $G_{\mc{F}}$-orbit of $U_x$, meaning that $G_{(\mc{P})}$ cannot be compact.
	
	
	Lemma~\ref{lem:basic_observations}(ii) already shows that (ii) implies (iii), and clearly (iii) implies (i).
\end{proof}

Since the subgroups $G_{(\mc{P})}$ as $\mc{P}$ ranges over the clopen partitions of $X$ form a base of neighbourhoods of the identity in $G$, we have the following corollary.

\begin{corollary}\label{cor:full_lc_implies_ample}
Let $G\leq \Homeo(X)$ be a piecewise full group equipped with the compact-open topology.  Then the following are equivalent:
\begin{enumerate}[(i)]
\item $G$ is locally compact;
\item $G$ is uniformly discrete;
\item $G$ is discrete.
\end{enumerate}
\end{corollary}

\begin{proof}
Each of the three conditions (i)--(iii) is equivalent to the statement that there exists a clopen partition $\mc{P}$, such that the same-numbered condition of Theorem~\ref{thm:full_lc_implies_ample} holds.  The equivalence of (i)--(iii) is then clear.
\end{proof}

Uniformly discrete groups are very far from being simple. 
In fact, they are residually finite, because they act on $X$ with finite orbits. 
\begin{lemma}\label{prop:finite_orbit_implies_resinf}	
Let $Y$ be a set and let $G\leq \Sym(Y)$. 
If all orbits of $G$ on $Y$ are finite, then $G$ is residually finite.
\end{lemma}
\begin{proof}
For each $y\in Y$, the pointwise fixator $\Fix(Gy)$ of the $G$-orbit $Gy$ of $y$ is a normal subgroup of finite index in $G$. 
Given any non-trivial $g\in G$, there is some $y\in Y$ such that $gy\neq y$, so $g\notin \Fix(Gy)$. 
\end{proof}

We shall presently see that being uniformly discrete imposes some algebraic conditions on a group. 
However, uniform discreteness is a \emph{dynamical condition}: it passes to subgroups but is not preserved by group isomorphisms. 
An example showing this is given in Section \ref{sec:example}, where the same group is shown to have two faithful actions on the Cantor set, one uniformly discrete, the other not. 
Indeed, the second action is obtained as a quotient of the first one, showing just how delicate the dynamical condition can be.

A further property of uniformly discrete piecewise full groups is that they are ample%
\footnote{Presumably Krieger chose the word ``ample" because they are particular cases of the topological analogue of ``full groups" introduced by Dye in the context of ergodic theory. } %
 in the sense  introduced by Krieger \cite{krieger}.  We paraphrase the definition.

\begin{definition}\label{def:ample}
	 A subgroup $G\leq\Homeo(X)$ is \defbold{ample} if it is piecewise full, locally finite, countable, and for every $g\in G$, the set of fixed points $\Fix(g)$ is clopen in $X$.
\end{definition}

\begin{proposition}\label{prop:ud_implies_ample}
Let $G\leq \Homeo(X)$ be uniformly discrete. Then $\Full(G)$ is ample.
\end{proposition}

\begin{proof}
	Let $G \le \Homeo(X)$ and suppose that $G$ is uniformly discrete with respect to $\mc{P}$. 
	Note (Lemma \ref{lem:basic_observations}(i)) that $\Full(G)$ is also uniformly discrete and has the same orbits on $X$ as $G$. 
	 From now on we  assume $G = \Full(G)$.
	
	Let $a\in X/G$ be a $G$-orbit on $X$ and let $\sigma_a:G \rightarrow \Sym(a)$ be the action of $G$ on $a$. 
	The orbits of $G$ on $X$ form a partition of $X$, and $G\leq \Homeo(X)$, so the homomorphism 
	$$G \rightarrow \prod_{a\in X/G} \Sym(a), \qquad g\mapsto (\sigma_a(g))_{a}$$ is injective. 
	Since $G$ is uniformly discrete, each $G$-orbit $a$ contains at most one point from each part of $\mc{P}$ and therefore $\Sym(a)$ can be identified with a canonical subgroup of $\Sym(\mc{P})$. 
	This gives an embedding of $G$ into the Cartesian product  $\Gamma:=\prod_{X/G}\Sym(\mc{P})$.
	
	Therefore $G$ is locally finite if $\Gamma$ is, which follows from a well-known argument: 
	Take $n\in\N$ and $g_1,\dots,g_n\in \Gamma$.
	For each $a\in X/G$, denote by $\pi_a$ the projection of $\Gamma$ onto the $a$th copy of $\Sym(\mc{P})$.
	There are $|\mathcal{P}|!$ possible values of $\pi_a(g_i)$ for each $i\in{1,\dots, n}$ and therefore 
	$m=(|\mc{P}|!)^n$ possible values for the $n$-tuple $(\pi_a(g_1), \ldots, \pi_a(g_n))$ as $a$ ranges over all $X/G$.
	For each $j\in\{1,\ldots,m\}$, denote by $H_j$ the subgroup of $\Sym(\mc{P})$ generated by the $j$th such $n$-tuple. 
	Then $\langle g_1,\ldots,g_n\rangle $ embeds in the finite direct product $H_1\times\dots\times H_m$ and is therefore finite, as required.

	To show that $G$ is countable, we use the fact that it is locally finite and hence a union of finite subgroups.
	Let $H<G$ be a finite subgroup of $G$ and suppose that $\mc{P}=\{U_1,\dots,U_n\}$. 
	Since $H$ is finite, the Boolean subalgebra of $\B$ generated by $\{hU_i\colon h\in H, i=1,\dots,n\}$ is finite and $H$-invariant. 
	Its atoms form a clopen partition $\mc{Q}$ which refines $\mc{P}$ and which is preserved by $H$. 
	Indeed, $H$ acts faithfully on the partition $\mc{Q}$: 
	if $h$ stabilises each part of $\mc{Q}$, then it also stabilises each part of $\mc{P}$ and the uniform discreteness of $G$ implies that $h$ must be trivial.
	In fact, uniform discreteness and the fact that $\mc{Q}$ refines $\mc{P}$ imply that the setwise stabiliser of $V\in\mc{Q}$ coincides with its pointwise stabiliser. 
	We conclude that every finite subgroup of $G$ acts faithfully as  a permutation group on some finite refinement $\mc{Q}$ of $\mc{P}$. There are finitely many possibilities for such permutation groups and fixed $\mc{Q}$ and countably many possibilities for $\mc{Q}$. 
	Therefore $G$ must be countable. 
	
	It remains to show, given $g \in G$, that $\Fix(g)$ is clopen.  In fact we need only show that $\Fix(g)$ is open; the fact that $\Fix(g)$ is closed follows from the fact that $g$ is a homeomorphism.
	Suppose that $g\in G$ fixes a point $x\in X$ and let $U$ be a clopen neighbourhood of $x$ which is entirely contained in some part of $\mc{P}$. 
	Since $x$ is fixed by $g$, the intersection $V_x:=gU\cap U$ is a non-empty clopen subset of $X$. 
	Given any $y\in V_x\setminus\{x\}$, there exists $z\in U\setminus\{x\}$ such that $y=gz$. 
	Since $G$ is uniformly discrete and $y,z\in U$ are in the same part of $\mc{P}$, we must have $y=z$.
	Thus $g$ fixes the clopen neighbourhood $V_x$ of $x$. 
	We therefore obtain that $\Fix(g)=\bigcup_{x\in \Fix(g)}V_x$ is open, as required. 	
\end{proof}

In particular,  uniformly discrete groups are locally finite and residually finite, which makes them LERF, or subgroup separable (every finitely generated subgroup is the intersection of finite-index subgroups that contain it).

The converse of Proposition \ref{prop:ud_implies_ample} is not true, since not every ample group is uniformly discrete.
For example, the direct union $\varinjlim_n \Sym (2^n)$ of symmetric groups on levels of the rooted binary tree is an ample group but is not uniformly discrete, since each point of the boundary of the tree has an infinite orbit. 
For those familiar with Thompson's group $V$, this is the subgroup consisting of those elements of $V$ that preserve the standard probability measure on the Cantor set, or equivalently (if $V$ is defined in terms of piecewise linear transformations of the unit interval) the subgroup of elements of $V$ in which every segment has slope $1$.

 Moreover, the following example shows that an ample group can have uniformly bounded orbits on the Cantor set, without being uniformly discrete.

\begin{example}\label{ex:c2_not_ud}
	Consider the Cantor set $X$ obtained as right-infinite words over the alphabet $\{0,1\}$.
	For each $n \ge 0$, denote by $g_n$ the homeomorphism of $X$ that exchanges the prefixes $1^n00$ and $1^n01$, leaving the rest of $X$ fixed pointwise (where $1^n$ denotes the word of length $n$ all of whose letters are 1).  Let $g_{\infty}$ be the homeomorphism of $X$ that exchanges the prefixes $1^n00$ and $1^n01$ for all $n \ge 0$.
	Let $G_1 = \langle g_{\infty} \rangle$ and let $G_2=\Full(\langle g_n, n\in \N\rangle)$.
	Then $G_1$ and $G_2$ have the same orbits, all of which have size at most 2.
	The group $G_1$ is finite but $\Full(G_1)$ is not ample, since its set of fixed points is the singleton $\{1^{\infty}\}$, which is not open.  On the other hand, $G_2$ is ample.
	Neither $G_1$ nor $G_2$ is uniformly discrete: in both cases, any partition of $X$ must have a part containing all words starting with $1^n$ for some $n\in\N$ and we see that two such points lie in the same orbit of $G_i$ for $i=1,2$.
\end{example}

On the other hand, given a \emph{finite} group $F$ of homeomorphisms in which the elements have clopen fixed points, there is a uniformly discrete invariant partition $\mc{P}$ for the action of $F$; that is, $F$ permutes $\mc{P}$ in such a way that the setwise stabiliser of each part coincides with its pointwise stabiliser. 
 The argument is taken from the proof of \cite[Lemma 2.1]{krieger}:

\begin{lemma}\label{lem:finite_clopen}
Let $F$ be a finite subgroup of $\Homeo(X)$ such that $\Fix(f)$ is clopen for all $f \in F$. 
 Then there is an $F$-invariant partition $\mc{P}$ of $X$ such that the setwise stabiliser in $F$ of each part coincides with its pointwise stabiliser.
\end{lemma}

\begin{proof}
Given $f\in F$ of order $n$ and a divisor $p$ of $n$, denote by $A_p(f)$ the set of points in $X$ whose $f$-orbit has size exactly $p$; in other words, $$A_p(f)=\Fix(f^p)\setminus \bigcup_{q<p, q\mid p}\Fix(f^q),$$
	 which makes it plain that $A_p(f)$ is clopen. 
	Each $A_p(f)$ can be partitioned further into clopen subsets $A_{p,i}(f)$ for $i=0,\dots,p-1$ such that $f(A_{p,i}(f))=A_{p,i+1}(f)$ modulo $p$, for each $i$. 
	Thus we obtain a partition of 
	$$X=\bigsqcup_{p|n, i\in\mathbb{Z}/p\mathbb{Z}} A_{p,i}(f).$$
	Denote by $\B(f)$ the Boolean subalgebra generated by the above partition. 
	The atoms of the Boolean subalgebra generated by $\{h\B(f)\colon h,f\in F \}$ form a clopen partition $\mc{P}$ of $X$ that is $F$-invariant. 
	If a part $U$ of $\mc{P}$ is preserved by some $f\in F$, then $U$ is contained in some $A_{p,i}(f)$ and we have $U\subseteq A_{p,i+1}(f)\cap A_{p,i}(f)\neq \emptyset$. 
	This can only occur if $p=1$, that is, if $U\subseteq \Fix(f)$.
	Thus the setwise stabiliser in $F$ of each part of $\mc{P}$ coincides with its pointwise stabiliser.
\end{proof}

We will investigate further the connection between finite groups and uniformly discrete ample groups in Section~\ref{sec:finite_origin}.

\section{Uniformly discrete groups among ample groups}

We now address the issue of distinguishing uniformly discrete groups among ample groups and describing their structure as locally finite groups.
We start by giving an algebraic description, by adapting an argument from \cite[Lemma 2.1]{krieger}.

Given a Boolean subalgebra $\mc{A}$ of $\mc{B}$, call a group $G\leq \Homeo(X)$ ($\cong\Aut(\B)$) \defbold{piecewise full on $\mc{A}$} if it leaves $\mc{A}$ invariant and $G = \Full_{\mc{A}}(G)$.

\begin{proposition}[See Lemma 2.1 of \cite{krieger}]\label{prop:udfull_direct_union_symmetrics}
Let $G\leq \Homeo(X)\cong\Aut(\B)$ be an ample group. 
Given any decomposition of  $G=\bigcup_{n\in\mathbb{N}}H_n$ as a direct union of finite subgroups $H_n$, 
there exist finite subgroups $G_n\leq G$ and finite subalgebras $\mc{B}_n\leq \B$ such that for each $n\in \N$:
\begin{enumerate}
\item $G_n\geq H_n$, 
\item $\mc{B}_n$ is $H_n$-invariant,
\item $G_n$ is piecewise full on $\mc{B}_n$ and the setwise stabiliser in $G_n$ of any atom of $\mc{B}_n$ coincides with the pointwise stabiliser,
\item $G_n=\bigoplus_{i=1}^{m_n}\Sym(\mc{O}_i)$ where the $\mc{O}_i$ range over the $H_n$-orbits on $\atom(\mc{B}_n)$,
\item the embeddings $G_n\hookrightarrow G_{n+1}$ are block-diagonal: for each $H_n$-orbit $\mc{O}$, the factor $\Sym(\mc{O})$ embeds diagonally in $\bigoplus_{j=1}^{s} \Sym(\mc{O}_j)\leq \bigoplus_{i\in I}\Sym(\mc{Q}_i)$ where $I\subseteq \{1,\dots,m_{n+1}\}$. 
Each $\mc{O}_j$ is a $H_n$-orbit on $\atom(\mc{B}_{n+1})$, permutation-isomorphic to $\mc{O}$; in turn, each $\mc{O}_j$ is contained in some $H_{n+1}$-orbit $\mc{Q}_i$, inducing the natural embedding $\Sym(\mc{O}_j)\leq \Sym(\mc{Q}_i)$.
\end{enumerate}

Moreover, if $G$ is uniformly discrete with respect to the clopen partition $P=\{U_1, \ldots, U_k\}$, then $G_n$ and $\mc{B}_n$ can be found such that $\atom(\mc{B}_n)$ is a refinement of $P$ and each $\mc{O}_i$ consists of at most $k$ atoms, each in a different part of $P$.%
\end{proposition}

\begin{proof}
We can assume $H_0 = \{1\}$.  We proceed inductively, starting with $G_0=\{1\}$ and $\mc{B}_0=\{X,\emptyset\}$.  If $G$ is uniformly discrete, fix a clopen partition $P = \{U_1,\dots,U_k\}$ of $X$ with respect to which $G$ is uniformly discrete.

Suppose that suitable $G_n$ and $\B_n$ have been found.
By Lemma~\ref{lem:finite_clopen} there is an $H_{n+1}$-invariant partition $\mc{P}_{n+1}$, such that for each part, the setwise stabiliser in $H_{n+1}$ coincides with the pointwise stabiliser.
If $G$ is uniformly discrete, we also choose $\mc{P}_{n+1}$ to be a refinement of $P$. 
Take $\B_{n+1}$ to be the Boolean algebra generated by $\mc{P}_{n+1}$ and all $H_{n+1}$-translates of $\B_n$.
By construction, $\atom(\B_{n+1})$ is an $H_{n+1}$-invariant clopen partition of $X$ that refines both $\mc{P}_{n+1}$ and $\atom(\B_n)$ (and also $P$ if $G$ is uniformly discrete).
This implies that the setwise stabiliser in $H_{n+1}$ of an atom of $\B_{n+1}$ must in fact be its pointwise stabiliser.

Since $G$ is piecewise full on $X$, in particular it contains $G_{n+1} := \Full_{\B_{n+1}}(H_{n+1})$.
The group $G_{n+1}$ is finite and inherits from $H_{n+1}$ the property that the setwise stabiliser of an atom of $\B_{n+1}$ is its pointwise stabiliser.

To see the fourth item, consider the orbits $\mc{O}_1, \dots, \mc{O}_{m_{n+1}}$ of $H_{n+1}$ on $\atom(\B_{n+1})$. 
Given such an orbit $\mc{O}_i$, and some $h\in H_{n+1}$ taking $V \in \mc{O}_i$ to another element $W\in\mc{O}_i$, the piecewise full group $G_{n+1}$ contains the ``transposition''
$$g(x)=\begin{cases}
h(x), \text{ if } x\in V,\\
h^{-1}(x), \text{ if } x\in W=h(V),\\
x, \text{ else }
\end{cases}$$
that only swaps $V$ and $W$. 
Thus $G_{n+1}$ contains $\Sym(\mc{O}_i)$ for each $(H_{n+1})$-orbit $\mc{O}_i$ and, since these orbits are disjoint, $G_{n+1}\geq \bigoplus_{i=1}^{m_{n+1}}\Sym(\mc{O}_i)$. 
The fact that $G_{n+1}$ is generated by its subgroups $\Sym(\mc{O}_i)$ for $1 \le i \le m_{n+1}$ then follows from the construction of $\B_{n+1}$ and $G_{n+1}$ in terms of $H_{n+1}$.
If $G$ is uniformly discrete, each orbit consists of at most $k$ atoms, since each one must be in a different part of $P$.

Let us now see why the embeddings are block-diagonal. 
Let $\mc{O}$ be an $H_n$-orbit on $\atom(\B_n)$ (it is also a $G_n$-orbit). 
Then $\mc{O}$ consists of atoms $A_1,\dots, A_r$ of $\B_n$
(in the uniformly discrete case, each $A_i$ is contained in a different part of $P$, so after relabelling we may assume that $A_i\subseteq U_i$, $1 \leq i\leq r\leq k$).
Say $A_i = h_iA_1$ for $h_i \in H_n$.  
In turn, each $A_i$ is a join of a subset $\{A_{i,1}, \ldots, A_{i,s}\}\subseteq \atom(\B_{n+1})$, and since $\atom(\B_{n+1})$ is $H_n$-invariant, we can label these atoms so that $A_{i,j} = h_iA_{1,j}$. 
In particular, $s$ does not depend on $i$ and $A_{i,j}$ is in the same $H_{n+1}$-orbit as $A_{1,j}$ for all $i$ and $j$.
Denote by $\mc{O}_j$ the $H_{n}$-orbit of $A_{i,j}$ for $1\leq j\leq s$ and $1\leq i \leq r$. 

By (iii),  the setwise stabiliser of each $A_i$ in $\Sym(\mc{O}) \le G_n$ is equal to its point stabiliser.
This is clearly also true for the subsets $A_{i,j}$. 
Indeed, we see that the stabiliser of $A_i$ in $\Sym(\mc{O})$ is the same as the stabiliser of $A_{i,j}$ in $\Sym(\mc{O})$ for $1 \le j \le s$. 
Thus for each $1 \le j \le s$, the action of $\Sym(\mc{O})$ on $\mc{O}_j$ is permutationally equivalent to its action on $\mc{O}$. 
At the same time, $\Sym(\mc{O})$ clearly fixes pointwise any atom of $\B_{n+1}$ outside of $\bigcup^s_{j=1}\mc{O}_j$, showing that $\Sym(\mc{O})$ is embedded in $\bigoplus_{j=1}^{s}\Sym(\mc{O}_j)$ as a diagonal subgroup.

Now, each $\mc{O}_j$ is contained in some $H_{n+1}$-orbit $\mc{Q}_j$ on $\atom(\mc{B}_{n+1})$. 
If $G$ is uniformly discrete and $1\leq j<j'\leq s$, then $\mc{O}_j$ and $\mc{O}_j'$ must be in different $H_{n+1}$-orbits, since they are in different parts of $P$, yielding the embedding $\Sym(\mc{O})\hookrightarrow \bigoplus_{j=1}^{s}\Sym(\mc{O}_j)\leq \bigoplus_{j=1}^{s}\Sym(\mc{Q}_j)\leq \bigoplus_{j=1}^{m_{n+1}}\Sym(\mc{Q}_j).$

If $G$ is not uniformly discrete then several $\mc{O}_j$ may lie in a single $H_{n+1}$-orbit on $\atom(\mc{B}_{n+1})$. 
Suppose that $\mc{O}_j$ and $\mc{O}_{j'}$ lie in the same $H_{n+1}$-orbit, $\mc{Q}_j$. 
Then the induced embedding
$\Sym(\mc{O})\hookrightarrow \Sym(\mc{O}_j)\oplus \Sym(\mc{O}_{j'})\leq \Sym(\mc{Q}_j)$
is diagonal in the sense that $\Sym(\mc{O})$ embeds diagonally into $\Sym(\mc{O}_j)\oplus \Sym(\mc{O}_{j'})$ with isomorphic actions, and 
this is preserved by the natural embedding into the larger $\Sym(\mc{Q}_j)$.
\end{proof}

Thus ample groups and, in particular, uniformly discrete piecewise full groups, are direct limits of direct products of symmetric groups. 
These groups, and versions with alternating groups instead of symmetric groups, have been studied in \cite{DudkoMedynets}, \cite{HartleyZalesskii} and \cite{lavnek}, respectively. 
However, the focus there is on simple groups, as seems to be the case with much of the literature on direct limits of symmetric or alternating groups (see references in cited items).

In the rest of this section, we translate the above description into the language of Bratteli diagrams and then dimension ranges and use it to distinguish uniformly discrete groups within the class of ample groups.

\subsection{Bratteli diagrams}\label{ssec:Bratteli}

Bratteli diagrams are graphs that provide an intermediate (and usually easy to describe) step between ample groups and the algebraic invariant encoding the orbit system -- the dimension range, as considered in \cite{krieger}, which we shall only deal with briefly here.  
These objects are very familiar to operator algebraists and scholars of Cantor dynamics, as they can be used to classify AF (approximately finite-dimensional) $C^*$-algebras and the dynamical systems that can be associated to them.
A further advantage of considering Bratteli diagrams is that the uniform discreteness condition has a natural combinatorial translation in the diagram. 

For the benefit of the uninitiated, and to set notation, we recall how to take this intermediate step from an ample group and, in Section \ref{ssec:dimension_range}, how to obtain the dimension range of the ample group from it. 
Much of the notation and terminology here follows that used in \cite{lavnek}.

Suppose that $G\leq \Aut(\B)$ is an ample group with $G=\bigcup_{n\in \N} G_n$ and $\B=\bigcup_{n\in \N} \B_n$ where $G_n$ is finite and piecewise full on the finite subalgebra $\B_n\leq \B$ (as established in Proposition~\ref{prop:udfull_direct_union_symmetrics});
 $G_0$ is trivial and $\B_0$ is the 2-element Boolean algebra, corresponding to the clopen subsets $X$ and $\emptyset$.
To this situation we associate the following $\N$-graded graphs $\tilde{B}=(\tilde{V},\tilde{E},s,r)$ and $B=(V,E,s,r,d)$:
\begin{itemize}
	\item $\tilde{V}=\bigsqcup_{n\in\N}\tilde{V}_n$ where $\tilde{V}_n = \atom(\B_n)$.
	\item $\tilde{E}=\bigsqcup_{n\geq 1}\tilde{E}_n$ where $\tilde{E}_n$ consists of edges that represent containment of atoms, determined by \defbold{source} $s\colon \tilde{E}_n\rightarrow \tilde{V}_{n-1}$ and \defbold{range} $r\colon \tilde{E}_n\rightarrow \tilde{V}_n$ maps. 
	That is, if the atom corresponding to $\tilde{v}_n$ is contained in that corresponding to $\tilde{v}_{n-1}$ there is a unique edge $\tilde{e}_n\in \tilde{E}_n$ such that $s(\tilde{e}_n)=\tilde{v}_{n-1}$ and $r(\tilde{e}_n)=\tilde{v}_n $.
	\item  For each $n$, denote by $\pi_n\colon \tilde{V}_n \rightarrow V_n$ the quotient of $\tilde{V}_n$ by the induced action of $G_n$, coming from that on $\B_n$. 
	Since $G_n$ also acts on $\B_{n+1}$, the map $\pi_n$ induces a quotient map on the edges
	$\pi_{n+1}\colon \tilde{E}_{n+1} \rightarrow E_{n+1}$  where $\pi_{n+1}(\tilde{e})=\pi_{n+1}(\tilde{f})$ if and only if $\pi_n(s(e))=\pi_n(s(f))$ and $\pi_n(r(e))=\pi_n(r(f))$.
	By abuse of notation, this last $\pi_n$ is the orbit quotient map of $G_n$ on $\tilde{V}_{n+1}$.
	\item For each $n$ and $v\in V_n$, denote by $d(v)$ the size of $\pi_n^{-1}(v)$ (i.e. the number of atoms of $\B_n$ in the $G_n$-orbit corresponding to $v$).
\end{itemize}
The graph $B$ is the \defbold{Bratteli diagram} associated to  $(G=\bigcup_n G_n, \B=\bigcup_n \B_n)$ while $\tilde{B}$ is the \defbold{extended Bratteli diagram} associated to the same pair. 

Figure \ref{fig:brattelieg} shows an example of  the first few levels of a Bratteli diagram and its corresponding extended Bratteli diagram.
\begin{figure}
\includegraphics{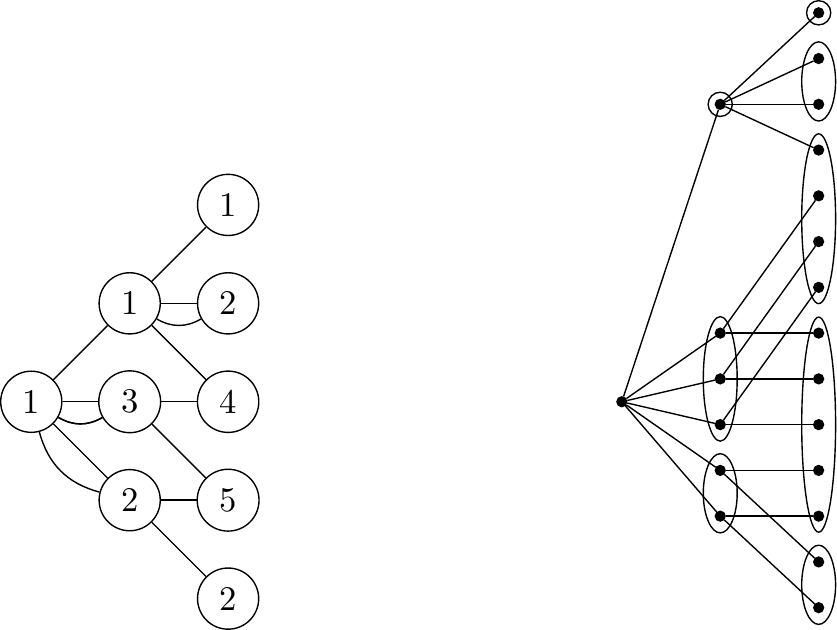}
\caption{First levels of a Bratteli diagram and its corresponding extended Bratteli diagram.}
\label{fig:brattelieg}
\end{figure}
Note that edges are directed and therefore so are paths: a path will always go from smaller $V_m$ or $\tilde{V}_m$ to larger $V_n$ or $\tilde{V}_n$.
Note also that, since the atoms of $\B_n$ are disjoint, no atom of $\B_{n+1}$ can be contained in two different atoms of $\B_n$; in particular, $\tilde{B}$ is a tree and each ultrafilter of $\B$ is uniquely given by an infinite path $(\tilde{v}_0, \tilde{v}_1, \tilde{v}_2, \tilde{v}_3, \dots)$ in $\tilde{B}$. 
Stone correspondence and the fact that $\B=\bigcup_{n\in\N} \B_n$ then yield a bijective correspondence between the infinite paths in $\tilde{B}$ and the points of $X$.
The topology on the set of infinite paths whose base consists of all paths starting from each vertex $\tilde{v}$ coincides with the topology on $X$. 

Notice that if we are just given a Bratteli diagram $B=(V,E,s,r,d)$, we can obtain its corresponding extended Bratteli diagram by choosing sets $\tilde{V}_n$ of size $\sum_{v\in V_n}d(v)$,  surjections $\pi_n\colon \tilde{V}_n\rightarrow V_n$ and $\tilde{E}_{n+1}=\{(\tilde{v}_n,e_{n+1})\in\tilde{V}_{n}\times E_{n+1} \mid s(e)=\pi_n(\tilde{v}_n)\}$.
Assuming, as we do throughout the paper,  that $d(v)=\sum_{v=r(e)}d(s(e))$ for all $v\in V\setminus\{v_0\}$, there is a bijection, which can be taken to be a range map,  $r\colon \tilde{E}_{n+1}\rightarrow \tilde{V}_{n+1}$, such that $r((\tilde{v}_n,e_{n+1}))=\tilde{v}_{n+1}$ implies that $r(e)=\pi_{n+1}(\tilde{v}_{n+1})$.

One then obtains an ample group acting on the space of infinite paths of $\tilde{B}$ as follows: for each $n\in\N$, put $G_n=\bigoplus_{v\in V_n}\Sym(\pi_n^{-1}(v))$ and let it act also on $\tilde{E}_{n+1}$  by acting on the first entry of $(\tilde{v_n},e_{n+1})$. 
Since the range map $r\colon \tilde{E}_{n+1}\rightarrow \tilde{V}_{n+1}$ is an embedding, it induces an embedding $G_n\hookrightarrow G_{n+1}$. 
The direct limit $G=\varinjlim_{\N} G_n=\bigcup_{\N} G_n$ is an ample group of homeomorphisms of the infinite paths of $\tilde{B}$ (which is homeomorphic to  $X$ if $d(v)\geq 1$ for every $v\in V$). 

 Proposition~\ref{prop:udfull_direct_union_symmetrics} implies that, taking the Bratteli diagram associated to the decomposition $G=\bigcup_{\N} G_n$, $\B=\bigcup_{\N}\B_n$ given there, and then taking the ample group of that Bratteli diagram as described above returns $G$, with the same decomposition $G=\bigcup_{\N} G_n$, $\B=\bigcup_{\N}\B_n$.

There is a natural notion of equivalence between Bratteli diagrams, called \defbold{telescoping}.
Given a Bratteli diagram $B=(V,E,s,r,d)$ and a strictly increasing sequence $(m_n)_n$ of natural numbers, the telescoping of $B$ along this sequence is the graph $B'=(V',E',s,r,d')$ where $V'=\bigcup_{n}\left(V'_n=V_{m_n}\right)$, $d'$ is the restriction of $d$ to $V'$ and $E'_n$ consists of all paths $e_1e_2\dots e_{m_n-m_{n-1}}$ in $B$ starting at $V_{m_{n-1}}$ and ending at $V_{m_n}$.
The source of such a path is $s(e_1)$ and its range $r(e_{m_n-m_{n-1}})$. 
We will consider Bratteli diagrams up to the equivalence relation generated by telescoping.

If $G\leq\Homeo(X)$ is an ample group, taking  different decompositions $G=\bigcup G_n=\bigcup H_n$, $\B=\bigcup \B_n=\bigcup \mc{C}_n$ produces  equivalent Bratteli diagrams, because for each $n\in \N$ we have  $\B_n\leq \mc{C}_{m_n}\leq \B_{k_{m_n}}$ and $G_n\leq H_{m_n}\leq G_{k_{m_n}}$ for some sequences $m_n$ and $k_{m_n}$.
The space of infinite paths of each extended Bratteli diagram naturally corresponds to $X$ and the induced dynamical systems all coincide. 

Similarly, if $B'$ is a telescoping of $B$, then their associated ample groups coincide, since one is a direct limit of a subsequence of groups of the other one.

\begin{theorem}\label{thm:bratteli_condition}
	Let $G\leq \Homeo(X)$ be an ample group with associated Bratteli diagram $B = (V,E)$. 
	Then $G$ is uniformly discrete if and only if there is a telescoping $B' = (V',E')$ of $B$ such that
	$B'\setminus (V'_0\cup E'_1)$ is a multitree (that is, there are no multiple directed paths between any pair of vertices $v,u\in V'\setminus V_0$).
\end{theorem}
\begin{proof}
	Without loss of generality, we can assume that the Bratteli diagram $B$ is obtained from the decomposition $G=\bigcup_n G_n$, $\B=\bigcup_n \B_n$ as in Proposition \ref{prop:udfull_direct_union_symmetrics}.
	
	Suppose first that $G$ is uniformly discrete and that $X=\bigsqcup_{i=1}^k U_i$ is a partition witnessing this. 
	Note that  there is some $n$ such that the atoms of $\B_n$ form a refinement of $X=\bigsqcup_{i=1}^k U_i$. 
	Telescope the diagram if necessary to assume that $n=1$.

	Suppose that there are two paths in $B$ between $v\in V_m$ and $u\in V_l$ for $l> m\geq 1$. 
	This means that there is a vertex $\tilde{v}_m\in \pi^{-1}_m(v)$ and distinct vertices $\tilde{v}_l, \tilde{u}_l\in \pi^{-1}_l(u)$ such that the clopens corresponding to $\tilde{v}_l$ and $\tilde{u}_l$ are contained in the clopen corresponding to $\tilde{v}_m$. 
	By choice of $m$, the clopen corresponding to $\tilde{v}_m$ is contained in some part $U_i$, and therefore so are the clopens corresponding to $\tilde{v}_l$ and  $\tilde{u}_l$. 
	But $G$ is uniformly discrete with respect to $U_1, \dots, U_k$, so $\tilde{v}_l$ and  $\tilde{u}_l$ cannot lie in the same $G_l$-orbit, a contradiction. 
	
	Conversely, suppose that the Bratteli diagram $B$ associated to $G=\bigcup_n G_n, \B=\bigcup_n \B_n$ has 
	no multiple paths, excluding the root. 
	The space $X$ of infinite paths of $\tilde{B}$ has a natural topology whose basic clopens are the sets $P(\tilde{v})$ of infinite paths starting at $\tilde{v}$ as $\tilde{v}$ ranges through $\tilde{V}$. 
	Then $X=\bigsqcup_{\tilde{v}\in \tilde{V}_1}P(\tilde{v})$ is a clopen partition of $X$.
	To see that it is a partition of uniform discreteness, take some $x\in X$, which is a member of $P(\tilde{v})$ for some $\tilde{v}\in\tilde{V}_1$. 
	Suppose that there is some $g\in G$ with $gx\in P(\tilde{v})$. 
	Then $g\in G_n$ for some $n$, so $x\in P(\tilde{v}_n)$ and $gx \in P(\tilde{u}_n)$ for some $\tilde{v}_n, \tilde{u}_n\in \tilde{V}_n$ such that $\pi_n(\tilde{u}_n)=\pi_n(\tilde{v}_n)=v_n$.
	Since there is only one path in $B$ from $\pi_1(\tilde{v})$ to $v_n$, we have $\tilde{u}_n = \tilde{v}_n$.
	  By the same argument, for all $n' \geq n$ there is $\tilde{v}_{n'} \in \tilde{V}_{n'}$ such that $\{x,gx\} \subseteq P(\tilde{v}_{n'})$. 
	 Hence $gx = x$ and $|Gx \cap  P(\tilde{v})|=1$, as required. 
\end{proof}


The above condition on the Bratteli diagram also shows that uniformly discrete ample groups are residually finite and that their corresponding  dimension range (and also AF $C^*$-algebra) are residually finite-dimensional (for every non-trivial element there is a finite-dimensional quotient in which it has non-trivial image). 
We only give a detailed proof for the case of groups (in essence reproving Lemma~\ref{prop:finite_orbit_implies_resinf}). 
The same argument works for dimension ranges and AF-algebras, using the well-known 1-1 correspondence between order ideals of  dimension ranges,  ideals of AF-algebras (\cite[IV.5.1]{davidson_book}), and ideals of the corresponding Bratteli diagram (\cite[III.4.2]{davidson_book}). 

\begin{definition}
Let $B$ be a Bratteli diagram (\emph{not} an extended Bratteli diagram). 
A subset $S$ of $V(B)$ is an \defbold{ideal} of $B$ if it satisfies both of the following:
\begin{enumerate}
\item  If  $v=s(e)\in V(B)$ belongs to $S$ and $w=r(e)$ then $w\in S$.
\item Given $v \in V(B)$, if $w=r(e)\in S$ for all $e\in E(B)$ such that $s(e)=v$ then   $v\in S$.
\end{enumerate}
\end{definition}

In the case of ample groups, not all normal subgroups correspond to ideals in the Bratteli diagram, but the following lemma will suffice (compare the analogous \cite[III.4.4]{davidson_book} for AF-algebras).

\begin{lemma}\label{lem:ud_normalsgps_corresp_bratteli_ideals}
	Let $G$ be an ample group with associated Bratteli diagram $B$. 
	For every ideal $S$ of $B$ there is a normal subgroup $N\unlhd G$, such that the quotient $G/N$ has the form
 $$G/N = \varinjlim_{n\in \N}\bigoplus_{v\in V_n\setminus S}\Sym(\pi^{-1}(v)).$$
\end{lemma}
\begin{proof}
Let $S$ be an ideal of $B$ and write $\tilde{S}$ for the preimage of $S$ in the extended Bratelli diagram $\tilde{B}$.
The first condition in the definition of an ideal ensures that $\tilde{B} \setminus \tilde{S}$ is a subtree of $\tilde{B}$ containing the root.
Let $Y$ be the subspace of $X$ corresponding to the infinite paths in $\tilde{B} \setminus \tilde{S}$, or equivalently, the infinite paths in $\tilde{B}$ that do not pass through $\tilde{S}$.
Because on each level $n$, the set $\tilde{V}_n \cap \tilde{S}$ is a union of $G_n$-orbits, we see that $Y$ is a closed $G$-invariant subspace.
We can therefore restrict the usual action of $G$ on $X$ to obtain an action on $Y$ with some kernel $N$.

From the Bratelli diagram, we see that $G/N$ acts as the direct limit 
\[
\varinjlim_{n\in \N}\bigoplus_{G\tilde{v}\subseteq \tilde{V}_n \setminus \tilde{S}}\Sym(G\tilde{v}) = \varinjlim_{n\in \N}\bigoplus_{v\in V_n\setminus S}\Sym(\pi^{-1}(v)),
\]
where in the first sum we have one summand for each $G$-orbit on $\tilde{V}_n \setminus \tilde{S}$.
\end{proof}

\begin{proposition}\label{prop:ud_ample_resfin_diagramproof}
	If $G$ is a uniformly discrete ample group then $G$ is residually finite.
\end{proposition}
\begin{proof}
Let $B$ be a Bratteli diagram associated to the action of $G$ on the Cantor set $X$.
Assume without loss of generality, using Theorem \ref{thm:bratteli_condition}, that $B\setminus(V_0\cup E_1)$ is a multitree and let $G=\varinjlim \left(G_n=\bigoplus_{v\in V_n}\Sym(\pi^{-1}(v))\right)$ be the associated decomposition of $G$.

Let $g\in G$ be a non-trivial element. 
Then $g\in G_n$ for some $n$ and has non-trivial projection onto some summand $\Sym(\pi^{-1}(v))$ with $v\in V_n$. 
Let $\gamma$ be some infinite directed path in $B$ starting from $v$ and denote by $C$ the subgraph of $B$ spanned by $\gamma$ and all directed paths from the root to any vertex of $\gamma$.  
Note that, by the multitree condition, there is at most one such path going through each $u\in V_1$. 
Since $G$ is piecewise full, it is infinite, so there are infinitely many infinite directed paths in $B$ (not necessarily starting at the root) that are not in $C$. 

The definition of $C$ guarantees that $V(B\setminus C)$ is an ideal: (i) if $e \in E(B)$ is such that $s(u)\in V(B\setminus C)$ then no path starting from $r(u)$ can end on $\gamma$, so $r(u) \in V(B \setminus C)$; (ii) given $u \in V(B)$, if $w=r(e)\in V(B\setminus C)$ for all $e\in E(B)$ such that $s(e)=u$ then no path starting from $u$ can end on $\gamma$, so $u\in V(B\setminus C)$.

By Lemma~\ref{lem:ud_normalsgps_corresp_bratteli_ideals}, there is a normal subgroup $N$ of $G$ such that
$$G/N\cong \varinjlim \bigoplus_{u\in C\cap V_n } \Sym(\pi^{-1}(u)).$$
Note that, as $B\setminus (V_0\cup E_1)$ is a multitree, $\{d(u)=|\pi^{-1}(u)| \colon u\in C\}$ is finite, with maximum $k$ say. 
Thus $G/N\cong \Sym(k)$ is finite. 
The fact that the projection of $g$ onto the direct summand $\Sym(\pi^{-1}(v))\hookrightarrow \Sym(k)$ is non-trivial implies that $g$ has non-trivial image in $G/N$.
\end{proof}

The above, and the condition in Theorem \ref{thm:bratteli_condition} is to be contrasted with the condition for simplicity of a Bratteli diagram: for every infinite path $\gamma$ in the Bratteli diagram and every vertex $u$, there is a path starting at $u$ and ending at some vertex in $\gamma$. 
It is not hard to show (see \cite[Proposition 5.2]{lavnek}) that a Bratteli diagram is simple if and only if its associated ample group acts minimally (all orbits are dense) on $X$. 
Again, this is in stark contrast to the action of a uniformly discrete group. 

The ample group $G$ associated to  a simple Bratteli diagram as we have defined it is not necessarily simple. 
Indeed, Lemma \ref{lem:ud_normalsgps_corresp_bratteli_ideals} does not account for all normal subgroups, because the normal subgroup obtained by taking alternating groups instead of symmetric groups in the direct limit decomposition of $G$ does not appear as an ideal of $B$. 
However, it can be shown \cite[Theorem 3.1]{lavnek} that, apart from degenerate cases, an ample group of this alternating form is simple if and only its associated Bratteli diagram is simple.

\subsection{Dimension ranges or $K_0$ groups}\label{ssec:dimension_range}

This subsection is a translation of the results in Subsection \ref{ssec:Bratteli} to dimension ranges.

The \defbold{dimension range} of an ample group is an algebraic invariant associated to it in \cite{krieger}, which classifies ample groups up to conjugation in $\Homeo(X)$ (\cite[Corollary 3.6]{krieger}).
It is directly inspired by Elliott's dimension group of an AF-algebra (which is in fact the scaled $K_0$ group of the algebra, an ordered abelian group, with some extra information -- the scale) that he shows classifies these $C^*$-algebras \cite{elliott}\cite[IV.4]{davidson_book}.

In \cite{krieger}, a \textit{unit system} is defined to be a pair $(\mc{A},G)$, where $\mc{A}$ is a subalgebra of $\mc{B}$ and $G$ is a countable locally finite subgroup of $\Homeo(X)$ that is piecewise full on $\mc{A}$ and acts faithfully on $\mc{A}$, and such that for all $g \in G$, the fixed-point set of every element of $G$ is an element of $\mc{A}$.
The dimension range of an ample group $G=\bigcup_{n\in\N}G_n, \B=\bigcup_{n\in\N}\B_n$ is constructed by first constructing dimension ranges for the finite unit systems $(\B_n,G_n)$ as follows. 
Start with the set $\B_n/G_n$ of $G$-orbits on $\B_n$ and form the $\Z$-module generated by $\B_n/G_n$, with a relation $w - v - u$ for $u,v,w \in \B_n/G_n$ whenever there exist disjoint $\tilde{v}\in v$ and $\tilde{u}\in u$ such that $\tilde{v}\sqcup\tilde{u}\in w$.
Performing the Grothendieck group construction on this monoid then gives an abelian group $\B_n | G_n$.
Moreover, the inclusion order on $\B_n$ is inherited by $\B_n | G_n$ and preserved by addition, turning $\B_n | G_n$ into an ordered abelian group. 
Indeed, as an ordered abelian group, $\B_n | G_n$ is isomorphic to $(\Z^{|V_n|}, \Z_+^{|V_n|})$ where $\Z_+^{|V_n|}$ is the submonoid of non-negative elements of the group. 
Since $G_n, \B_n$ are finite, every element of $\B_n | G_n$ is a sum of the $G_n$-orbits of $\B_n$-atoms it contains. 
This corresponds to $K_0$ of the system $(\B_n,G_n)$. 
The extra information that we need is the tuple $(d(v))_{v\in \atom(\B_n)}$ consisting of the number $d(v)$ of elements in the $G_n$-orbit of $v\in \atom(\B_n)$. 
This is known as the \defbold{scale} $\Gamma_n$ of $(\B_n,G_n)$ and the triple $(\B_n | G_n, (\B_n | G_n)_+, \Gamma_n)$ is the \defbold{dimension range} of $(\B_n,G_n)$. 
To put it simply, the dimension range of  $(\B_n,G_n)$ is isomorphic to $(\Z^{|V_n|}, \Z_+^{|V_n|}, \{(i_v)_{v\in V_n}\colon 0\leq i_v\leq d(v)\})$. 

The inclusion $(\B_n,G_n)\leq (\B_{n+1},G_{n+1})$ induces a morphism of dimension ranges which is completely described by a $| \atom(\B_n)/G_n|$-by-$| \atom(\B_{n+1})/G_{n+1}|$ matrix whose $(i,j)$ entry is the number of atoms of the orbit $v_j\in \atom(\B_{n+1})/G_{n+1}$ that are contained in $v_i\in  \atom(\B_n)/G_n$. 
In other words, passing to the Bratteli diagram corresponding to $G=\bigcup_{n\in\N} G_n, \B=\bigcup_{n\in\N}\B_n$, the matrix just described is the adjacency matrix between $V_n$ and $V_{n+1}$. 
	
We can therefore define the dimension range of $(\B,G)$ as the direct limit of dimension ranges $(\B_n | G_n, (\B_n | G_n)_+, \Gamma_n)$ with the morphisms described above. 

Of course, one could start with a Bratteli diagram and build its corresponding ample group, the orbits of which would yield a dimension range. 

 With the Bratteli diagram dictionary at our disposal, Theorem \ref{thm:bratteli_condition}  easily yields a characterisation of the dimension ranges of ample groups that are uniformly discrete. 
 
 \begin{corollary}\label{cor:dimension_range}
 Let $G=\bigcup_{n\in\N}G_n, \B=\bigcup_{n\in\N}\B_n$ be an ample group with dimension range $\varinjlim (\B_n | G_n, (\B_n | G_n)_+, \Gamma_n)$. 
 Then $G$ is uniformly discrete if and only if the matrices describing all but finitely many of the morphisms in the direct limit have only 1s and 0s as entries.
 \end{corollary}

The analogue of Lemma \ref{lem:ud_normalsgps_corresp_bratteli_ideals} is a combination of III.4.2, III.4.4 and IV.5.1 of \cite{davidson_book} and is well-known.
An \defbold{order ideal} of a partially ordered group $(\Gamma, \Gamma_+)$ is a subgroup $\Delta\leq \Gamma$ such that, denoting $\Delta_+:=\Delta\cap \Gamma_+$, we have that $\Delta=\Delta_+-\Delta_+$ and if $0<\gamma<\delta$ for some $\gamma\in \Gamma$ and $\delta\in \Delta_+$, then $\gamma\in \Delta$.
\begin{lemma}\label{lem:Bratteli_ideal_corresp_order_ideal}
	The ideals of a Bratteli diagram $B$ associated to an ample group $G$ are in 1-1 correspondence with the order ideals of the dimension range of $G$. 
\end{lemma}

An analogous argument to Proposition \ref{prop:ud_ample_resfin_diagramproof} shows that dimension ranges of uniformly discrete ample groups are \defbold{residually finitely-generated} (every non-trivial element has non-trivial image in a finitely generated quotient).

\section{Not every uniformly discrete ample group arises from a finite group}\label{sec:finite_origin}

The most obvious examples of uniformly discrete piecewise full groups are obtained as piecewise full groups of finite homeomorphism groups, necessarily acting in a uniformly discrete way.  Given a finite group $F$, there are several equivalent ways to characterise when $\Full(F)$ is uniformly discrete.

\begin{proposition}\label{prop:F_is_ud}
Let $F\leq\Homeo(X)$ be finite and let $G=\Full(F)$. The following are equivalent:
\begin{enumerate}
	\item $F$ is uniformly discrete;
	\item $G$ is uniformly discrete; 
	\item $G$ is ample;
	\item for every $f\in F$, the set $\Fix(f)$ of fixed points of $f$ is a clopen subset of $X$;
	\item $F$ acts on some clopen partition $\mc{P}$ of $X$ such that the setwise stabiliser of each part coincides with its pointwise stabiliser. 
\end{enumerate}
\end{proposition}

\begin{proof}
	(i) and (ii) are equivalent by Lemma~\ref{lem:basic_observations}(i) and (ii) implies (iii) by Proposition~\ref{prop:ud_implies_ample}. Clearly (iii) implies (iv), and (iv) implies (v) by Lemma~\ref{lem:finite_clopen}.
	It is immediate that an action of $F$ on $\mc{P}$ as in (v) makes the action of $F$ on $X$ uniformly discrete with respect to $\mc{P}$, so (v) implies (i). 
\end{proof}

Accordingly, we say that a group $G < \Homeo(X)$ is \defbold{ample of finite origin} if $G = \Full(F)$ for a finite group and $G$ is ample (equivalently, $G$ is uniformly discrete).

Not all uniformly discrete ample groups are of finite origin.  Here are criteria to distinguish ample groups of finite origin from other uniformly discrete piecewise full groups.

\begin{proposition}\label{prop:full_finite_iff_top_bratteli}
Suppose that $G\leq\Homeo(X)$ is uniformly discrete and piecewise full.
The following are equivalent:
\begin{enumerate}
	\item $G$ is ample of finite origin. 
	\item Given a Bratteli diagram $B=(V,E)$ associated to $G$, there exists $N\in\N$ such that $d(v)=d(u)$ for every $v\in V_N$ and every vertex $u\in V$ on a directed path starting at $v$.
	\item For every clopen subset $U$ of $X$, the union $\bigcup_{g\in G}g(U)$ of translates of $U$ by $G$ is clopen.
\end{enumerate}
\end{proposition}

\begin{proof}

\textbf{(i) $\Rightarrow$ (ii):}
Let $G=\Full(F)$ for some finite $F<\Homeo(X)\cong\Aut(\B)$ acting uniformly discretely. 
Appealing to Theorem \ref{thm:bratteli_condition}, let $B=(V,E)$ be the Bratteli diagram associated to some decomposition $G=\bigcup_{n\in\N}G_n, \B=\bigcup_{n\in\N}\B_n$, chosen so that $B\setminus(V_0\cup E_1)$ is a multitree.

By Proposition~\ref{prop:F_is_ud}, there is some clopen partition $\mc{P}$ of $X$ on which $F$ acts in such a way that the setwise stabiliser of a part coincides with the part's pointwise stabiliser. 
Because $\B=\bigcup_{n\in\N}\B_n$, there is some $N\in\N$ such that $\atom(\B_N)$ is a refinement of $\mc{P}$. 
In particular, $F$ acts on $\atom(\B_N)$ and therefore also on $\atom(\B_m)$ for every $m\geq N$. 
This means that the vertices in $V_m$ are the $F$-orbits (=$G_m$-orbits)  on $\atom(\B_m)$ for each $m\geq N$.

Suppose for a contradiction that $d(u)>d(v)$ for some $v\in V_N$ and descendant  $u\in V_m$ of $v$ with $m\geq N$. 
In other words (and abusing notation) there is some $\tilde{v}_N \in \atom(\B_N)$ which contains some $\tilde{w}_m \in \atom(\B_m)$ such that the $F$-orbit of $\tilde{w}_m$ is larger than the $F$-orbit of $\tilde{v}_N$. 
Because $F$ preserves the Boolean subalgebras $\B_N$ and $\B_m$, the extra elements in the orbit of $\tilde{w}_m$ must be contained in the orbit of $\tilde{v}_N$, but this means that there are at least two paths between $v$ and $w$, contradicting the uniform discreteness criterion. 

\textbf{(ii) $\Rightarrow$ (iii):}
Suppose that $(V,E)$ is the Bratteli diagram associated to some decomposition $G=\bigcup_{n\in\N}G_n, \B=\bigcup_{n\in\N}\B_n$
 and let $U$ be a clopen subset of $X$. 
Then $U$ is the union of some atoms of $\B_n$, which correspond to vertices of $\tilde{V}_n$ in the extended Bratteli diagram. 
Taking a larger $n$ if necessary, assume that $n\geq N$ from the statement in (ii) and consider some $\tilde{v}_n\in\tilde{V}_n$ corresponding to an atom contained in $U$. 
Then, by the assumption in (ii), for each $m\geq n$,  the $G_n$-orbit of $\tilde{v}_n$ has the same size as the $G_m$-orbit of $\tilde{w}_m$ where $\tilde{w}_m\in\tilde{V}_m$ is a descendant of $\tilde{v}_n$. 
This implies that 
$$\bigcup_{\tilde{w}_m\leq \tilde{v}_n}\bigcup_{g\in G_m}g(\tilde{w}_m)=\bigcup_{g\in\G_m}g(\tilde{v}_n)=\bigcup_{g\in G_n}g(\tilde{v}_n)$$
for all $m\geq n$. 
Therefore $\bigcup_{g\in G}g(\tilde{v}_n)=\bigcup_{g\in G_n}g(\tilde{v}_n)$. 
The latter set is in $\B$ as it is a finite union of elements of $\B$. 
Since $U$ is the disjoint union of (finitely many) atoms $\tilde{v}_n$, we have that $\bigcup_{g\in G}g(U)$ is the union of finitely many clopen sets (elements of $\B$) and is therefore itself clopen too.

%
%
	
\textbf{(iii) $\Rightarrow$ (i)}:
	Suppose that  for every clopen $U$, the union $\bigcup_{g\in G}g(U)$ of all its translates by $G$ is clopen.
	Let $\mc{P}=\{U_1, \dots, U_k\}$ be a partition witnessing the uniform discreteness of $G$. 
	By Proposition~\ref{prop:ud_implies_ample}, the group $G$ is locally finite and can be written as a directed union $G=\bigcup_{\N}G_n$ of finite subgroups $G_n$ which act faithfully on a refinement $A_n$ of $\mc{P}$. 
	
	Then  $\{\bigcup_{g\in G_n} g(U_i) \colon n\in\N\}$ forms an open cover of $\bigcup_{g\in G} g(U_i)$, which we have assumed to be a clopen subset of $X$, and therefore compact. 
	This means that for each $i$ there is some $N_i\in\N$ such that 
	$$\bigcup_{g\in G_{N_i}} g(U_i)=\bigcup_{g\in G_m}g(U_i)=\bigcup_{g\in G}g(U_i)$$ for all $m\geq N_i$. 
	In particular, there is $N=\max \{N_i \colon i=1, \dots, k\}$ such that $\bigcup_{g\in G_N}g(U_i)=\bigcup_{g\in G}g(U_i)$ for every $i\in \{1, \dots, k\}$. 
	
	We claim that $G$ is the piecewise full group of $G_N$. 
	Let $g\in G=\bigcup_{n\in\N}G_n$, and let $n\in N$ be smallest such that $g\in G_n$. 
	If $n\leq N$ there is nothing to show, so suppose that $n>N$. 
	Let $\mc{P}_N=\{V_1, \ldots, V_{k_N} \}$ be the refinement of $\mc{P}$ on which $G_N$ acts faithfully, and $\mc{P}_n=\{W_1, \dots, W_{K_N}\}$ the refinement of $\mc{P}_N$ on which $G_n$ acts faithfully. 
	For each $i\in\{1, \ldots, k\}$ we have assumed that $g(U_i)\subseteq \bigcup_{h\in G_N}h(U_i)$. 
	Suppose that the parts of $\mc{P}_n$, respectively, $\mc{P}_N$, contained in $U_i$ are labelled by $i_n\subseteq\{1, \ldots, k_n\}$, respectively, $i_N\subseteq\{1,\ldots,k_N\}$.
	Then $$g(U_i)=\bigsqcup_{j\in i_n}g(W_j)\subseteq\bigcup_{h\in G_N}h(U_i)=\bigcup_{h\in G_N}\bigsqcup_{l\in i_N}h(V_l).$$
	Since $g$ preserves $A_n$, and $G_N$ preserves $A_N$ and $A_n$ refines $A_N$, for each $j\in i_n$ there is some $h_j\in G_N$ and $l\in i_N$ such that $g(W_j)\subseteq h_j(V_l)$. 
	Because both $W_j$ and $V_l$ are contained in $U_i$, the uniform discreteness of the $G$-action with respect to $\mc{P}$ implies that $W_j\subseteq V_l$ and $g\rest_{W_j}=h_j\rest_{W_j}$. 
	Repeating the argument for each $U_i$ yields that for each $W_j\in \mc{P}_n$ there exists $h_j\in G_N$ such that $g\rest_{W_j}=h_j\rest_{W_j}$, as required. 
\end{proof}

It is now easy to see that not all uniformly discrete ample groups are of finite origin.  We note the following special case, which is also illustrated by the group $G$ in Section~\ref{sec:example} below.

\begin{corollary}
Let $F < \Homeo(X)$ be a finite group acting uniformly discretely, let $x \in X$ and let $G$ be the stabiliser of $x$ in $\Full(F)$; suppose that $G \neq \Full(F)$.  Then $G$ is uniformly discrete and ample, but it is not of the form $\Full(H)$ for any finite group of homeomorphisms $H$.
\end{corollary}

\begin{proof}
Since $G$ is a subgroup of $\Full(F)$, it is uniformly discrete; by construction $G$ is also piecewise full, hence ample by Proposition~\ref{prop:ud_implies_ample}.
Since $G \neq \Full(F)$, there exists $f \in F$ such that $fx \neq x$.  By continuity there is a clopen neighbourhood $U$ of $x$ such that $U$ and $fU$ are disjoint.  We then see that
\[
U \cap \bigcup_{g \in G}g(f(U)) = U \setminus \{x\},
\]
so $\bigcup_{g \in G}g(f(U))$ is not closed, and thus $G$ does not satisfy condition (iii) of Proposition~\ref{prop:full_finite_iff_top_bratteli}.  Hence $G$ cannot be of the form $G = \Full(H)$ for any finite $H < \Homeo(X)$.
\end{proof}

One might wonder at this point whether all uniformly discrete ample groups arise as \emph{subgroups} of piecewise full groups of finite groups.  In fact, they do not all arise in this way; to explain why not, a more involved example is needed.

\begin{example}\label{eg:not_stab_finite_ud}
	For simplicity and concreteness, we will focus on a particular example, which can be easily generalised. 
	As in Example \ref{ex:c2_not_ud},  we consider $X$ as the boundary of the infinite binary rooted tree and more specifically as the infinite words in the alphabet $\{0,1\}$. 
	Let $$Y_1:=\{01^n0^{\infty} : n\in \N \}=\{0^{\infty}, 010^{\infty}, 0110^{\infty}, \dots \}\subset 0X=\{0w : w\in X \}$$
	and $Y_2=\{1^{\infty}\}\subset 1X$.
	Then $0X\setminus Y_1$ and $1X\setminus Y_2$ are both noncompact Cantor sets and therefore homeomorphic. 
	There are many possible homeomorphisms; we consider one, $h$, that induces a bijection of clopen subsets of $X$, described thus: enumerate the ``obvious'' clopen subsets $c_n$ of $X_1\setminus Y_1$, by short-lex order and ignoring those that are contained in one that has already been enumerated
	\begin{multline*}
	c_1=001X, c_2=0001X, c_3=0101X, c_4=0^41X, c_5=01001X, \\
	c_6=01101X,	 c_7=0^51X, c_8=010^31X, c_9=011001X, c_{10}=01^301X, \dots
	\end{multline*}
	The homeomorphism $h$ exchanges each clopen $c_n$ with $1^n0X$, for $n\geq 1$ in some way (which way exactly does not matter); see Figure \ref{fig:not_stab_finite_ud}.

\begin{figure}
	
\includegraphics[width=0.9\linewidth]{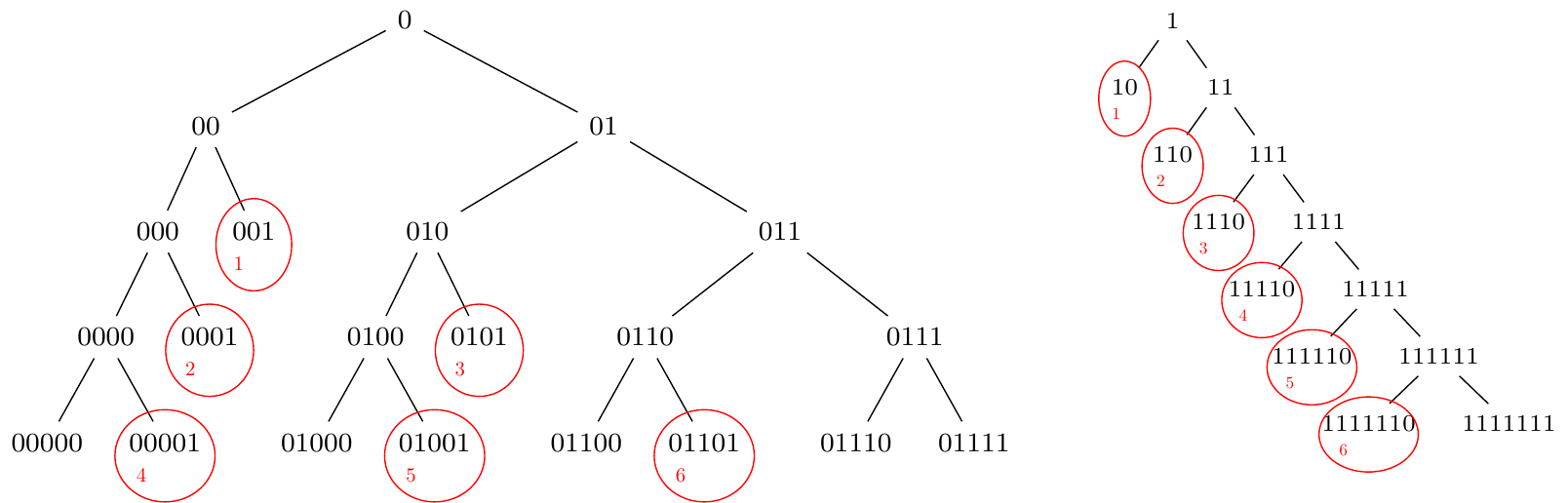}
	
\caption{The homeomorphism $h$ exchanges the clopens below the vertices  circled in red on the left with those with corresponding red number on the right.}
\label{fig:not_stab_finite_ud}

\end{figure}

	This homeomorphism obviously cannot be extended to all of $X$. 
	However we can take the group $G$ generated by homeomorphisms $h_U$ of the form
	\[
	h_U(x)=\begin{cases} h(x) &\text { if } x \in U\\
	h^{-1}(x) & \text{ if } x \in h(U)\\
	x &\text{ otherwise }
	\end{cases},
	\]
	where $U$ ranges over the compact open subsets of $0X \setminus Y_1$.
	By construction, $G$ is piecewise full and uniformly discrete (take $0X, 1X$ as a partition of uniform discreteness). 
	
	Suppose  that $G \leq \Full(F)$ for some finite $F<\Homeo(X)$.
	We will arrive at a contradiction by showing that $1^{\infty}$ must be taken to the points of $Y_1$, which cannot be achieved using only the finitely many elements from $F$. 
	
	Since $G$ is piecewise full, there are elements $g_n\in G$ such that $g_n$ swaps only $c_n$ and $1^n0X$ and fixes the rest of $X$, for $n\geq 1$. 
	Take a subsequence $(g_{n_m})_m\subset (g_n)_n$ such that $c_{n_m} \subseteq 0^m1X$ for $m \ge 2$.
	As $G\leq \Full(F)$, each $g_{n_m}$ is a gluing of finitely many restrictions of elements of $F$. That is, there is a finite partition $0^m1X=\bigsqcup_{i=1}^{r_m}U_{n_{m_i}}$ of $0^m1X$ into clopen sets such that $g_{n_m}|_{U_{n_{m_i}}}=f_{n_{m_i}}|_{U_{n_{m_i}}}$ for some elements $f_{n_{m_i}}\in F$. 
	As $F$ is finite and $(n_m)_m$ is infinite, some element, say $f_1\in F$, must appear infinitely often among the $f_{n_{m_i}}$. 
	The union of clopens $U_{n_{m_i}}$ corresponding to $f_1$ in the above setting has $0^{\infty}$ as its only boundary point, so $f_1(0^{\infty})$ must be the only boundary point of the union of  $f_1(U_{n_{m_i}})=g_{n_m}(U_{n_{m_i}})\subset 1^{n_m}0X$. 
	So $f_1(0^{\infty})=1^{\infty}$. 	
	Repeat this procedure for $01^l0^mX, l,m\geq 1$ to obtain that $f_l(01^l0^{\infty})=1^{\infty}$. 
	
	Now, $F$ is finite so there must be some repetition among the $f_l$; that is, $f_l=f_k$ for some $l\neq k$, which means that $f_l(01^l0^{\infty})=f_l(01^k0^{\infty})=1^{\infty}$, contradicting the assumption that $f_l$ is a homeomorphism of $X$. 
\end{example}


\begin{qu}\label{qu:which_ud_groups_finite_origin}
	Which uniformly discrete piecewise full groups arise as subgroups of ample groups of finite origin? 
	Can they be distinguished by their Bratteli diagrams?
\end{qu}

\section{Example: isomorphic groups, non-isomorphic dynamical systems}\label{sec:example}

We present an illustrative example of two different ample actions of the same group, one uniformly discrete, the other not.

Let $I_n$ be all $n$-digit  strings over the alphabet $\{0,1\}$,  except for the all-$1$ string $1^n$ (so $|I_n| = 2^n-1$); let $\Gamma_n$ be the elementary abelian group $\bigoplus_{i \in I_n}\grp{g_i}$, where each generator $g_i$ has order $2$.  
The group we are interested in is the direct limit $\Gamma=\varinjlim_{n\geq 1} \Gamma_n$, with diagonal embeddings:
$\Gamma_n$ embeds in $\Gamma_{n+1}$ by setting $g_i = g_{i0} + g_{i1}$ for all $i \in I_n$.

We consider two different embeddings $\alpha, \beta: \Gamma\rightarrow \Homeo(X)$  that give two different ample groups $G=\alpha(\Gamma)$ and $H=\beta(\Gamma)$, as follows. 
First consider $X$ as the boundary of the infinite binary rooted tree, thought of as the Cayley graph of the free monoid $\{0,1\}^*$.
In other words, we think of $X$ as the set $\{0,1\}^{\N}$ of (right) infinite words over the alphabet $\{0,1\}$.  The Boolean algebra $\B$ of clopen subsets of $X$ then has a standard decomposition $\B = \bigcup_n \B_n$, where $\B_n$ is the subalgebra with atoms $\{wX \colon |w| = n\}$.

Let $G$ be the stabiliser of the point $1^{\infty} \in X$ in $\Full(C_2)$ where $C_2$ acts on $X$ by swapping the first letter of every infinite word. 
Observe that $G$ and $\Full(C_2)$ are uniformly discrete with respect to the partition $0X \sqcup 1X$.
Define $\alpha: \Gamma \rightarrow G$ by letting $\alpha(g_i)$ swap $0iw$ and $1iw$ for all $w \in \{0,1\}^{\N}$.
Figure \ref{fig:Gaction} shows the first few levels of this action.

\begin{figure}[h]
	\begin{center}
	\includegraphics[width=0.9\linewidth]{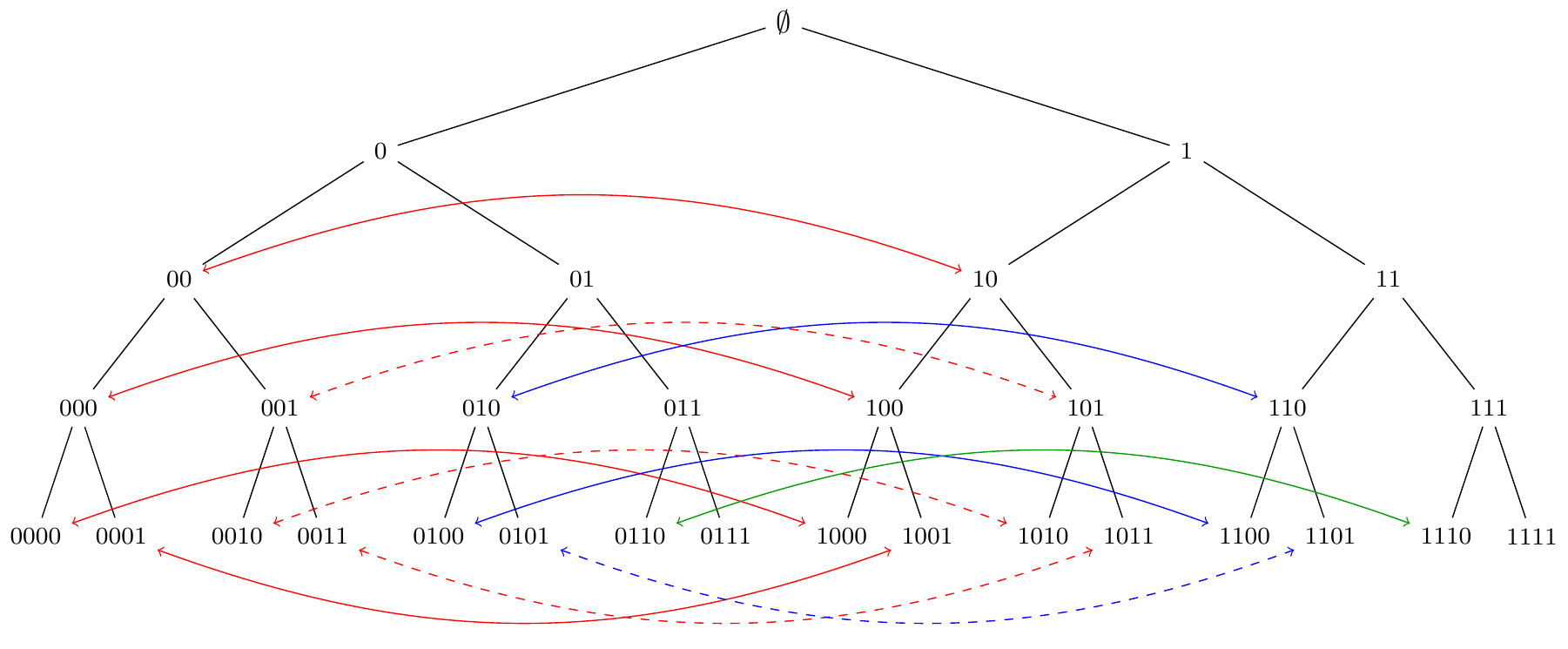}
	\end{center}
	\caption{Action of $G$ on first few levels of $\{0,1\}^*$. 
		The embeddings are colour-coded: ${\color{red}C_2}\hookrightarrow{\color{red}C_2\oplus C_2}\oplus{\color{blue}C_2}\hookrightarrow {\color{red}C_2\oplus C_2\oplus C_2\oplus C_2}\oplus {\color{blue}C_2\oplus C_2}\oplus {\color{green!60!black}C_2}$.
	}
	\label{fig:Gaction}
\end{figure}

We define $\beta$ on the generators $g_i$ as as follows: 
suppose $i = 1^m0j$ for some $m \ge 0$ and some (possibly empty) finite string $j$. 
 Then $\beta(g_i)$ swaps $1^m00jw$ and $1^m01jw$ for all $w \in \{0,1\}^{\N}$.

For $m \ge 0$, let $h_m$ be the homeomorphism of $X$ swapping $1^m00w$ and $1^m01w$ for every $w\in\{0,1\}^{\N}$.
  This time, we see that $\beta(\Gamma_n)$ is the piecewise full group of $\grp{h_m \mid 0 \le m \le n}$ acting on $\B_{n+2}$, and consequently $H = \beta(\Gamma)$ is  the piecewise full group of $\langle h_m\colon m\geq 0 \rangle$ acting on $X$.
  Also,  $H=\Full(\langle f\rangle)$ where $f=h_0h_1h_2\cdots $ is a homeomorphism of order 2  swapping $1^m00w$ and $1^m01w$ for every $w\in\{0,1\}^{\N}$ and every $m\in  \N$. 
The action on the first few levels is shown in Figure \ref{fig:Haction}.

\begin{figure}[h]
	
\begin{center}
	\includegraphics[width=0.9\linewidth]{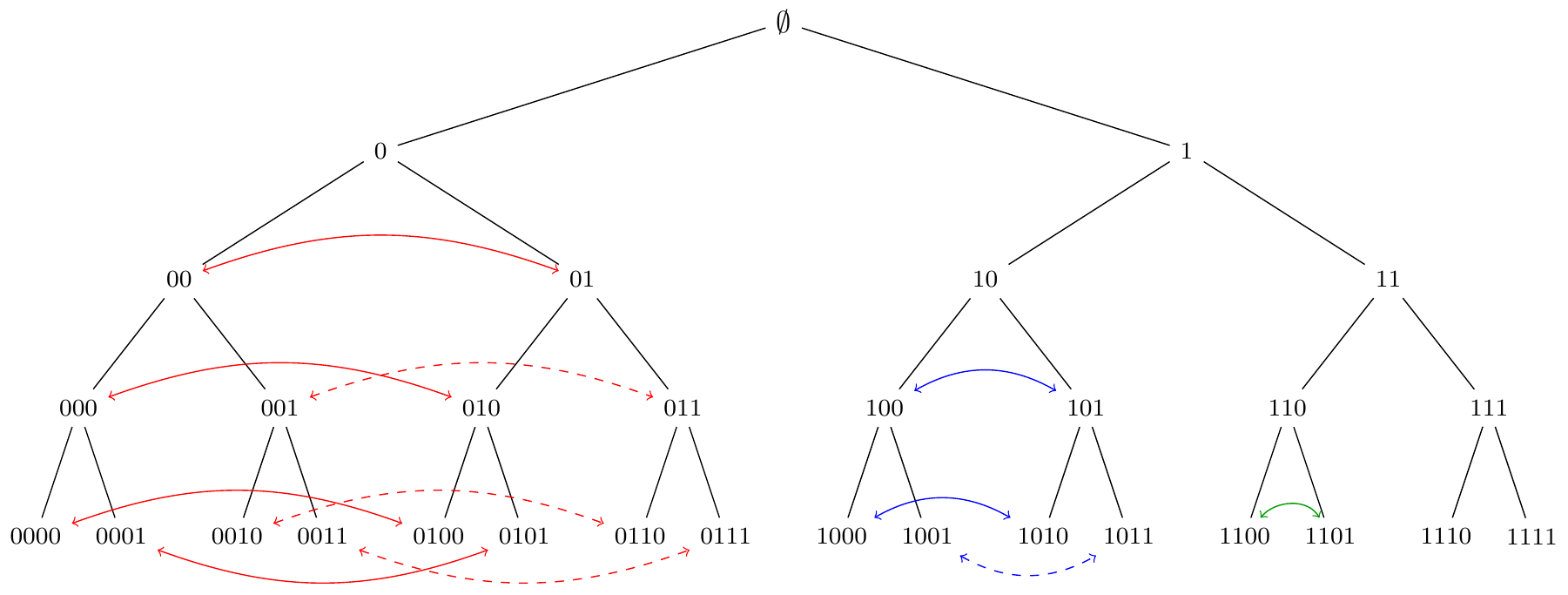}
\end{center}
	
	\caption{Action of $H$ on first few levels of $\{0,1\}^*$. 
		The embeddings are colour-coded: ${\color{red}C_2}\hookrightarrow{\color{red}C_2\oplus C_2}\oplus{\color{blue}C_2}\hookrightarrow {\color{red}C_2\oplus C_2\oplus C_2\oplus C_2}\oplus {\color{blue}C_2\oplus C_2}\oplus {\color{green!60!black}C_2}$.
	}
	\label{fig:Haction}
\end{figure}

We see that $G$ and $H$ are not conjugate in $\Homeo(X)$, for two reasons.  First, $G$ has two fixed points, namely $01^{\infty}$ and $1^{\infty}$, while $H$ only has one, namely $1^{\infty}$.  Second, $G$ is uniformly discrete, while $H$ is not, by Proposition~\ref{prop:F_is_ud}, using that $H=\Full(\langle f\rangle )$.

As should be expected (cfr. \cite[3.6]{krieger}), the Bratteli diagrams and dimension ranges associated to the different actions $G, H$ of $\Gamma$  are  different. 
Let us see them explicitly. 

For both $G$ and $H$, we telescope the standard decomposition $\B=\bigcup_n \B_n$ by removing the level $\B_1$.
Because we have skipped the first level of the tree, the extended Bratteli diagrams corresponding to $(G,\B)$ and $(H, \B)$ both have  $2^{n+1}$ vertices at level $n$, organised into $2^n-1$ orbits of size 2 and two orbits of size 1. 

\underline{For the action of $G$, }
$0u$ and $1u$ are in the same orbit for all $u\in\{0,1\}^{n-1} \setminus\{1^{n-1}\}$ while $01^{n-1}$ and $1^n$ are in orbits by themselves. 
%
%
%
%
%
%
This pattern repeats at all levels, from which it can be seen that there are no multiple paths in the Bratteli diagram.
Figure \ref{fig:G} shows the first few levels of the Bratteli diagrams of this action.

\begin{figure}[h]
\begin{center}	
	\includegraphics
{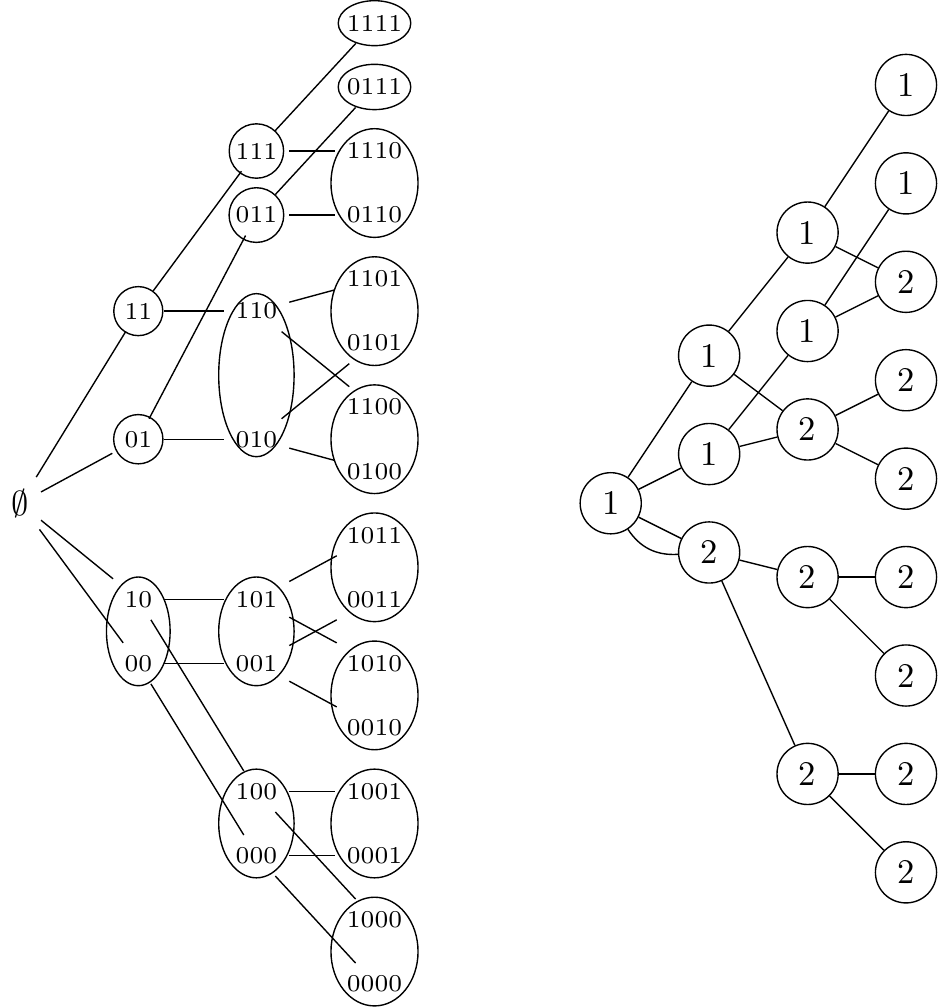}
\end{center}

	\caption{First levels of Bratteli diagrams of $G$}
	\label{fig:G}
\end{figure}

\underline{For the action of $H$}, the orbits of size 1 are  $1^n0$ and $1^n1$ while $1^m00w$ and $1^m01w$ are in the same orbit, where $m+|w|=n-1$ and $w\in\{0,1\}^{n-1-m}$. 
%
%
%
%
%
The pattern here also repeats at all levels, and we in fact have multiple paths between infinitely many pairs of vertices. 
The first few levels of the Bratteli diagrams of this action are shown in Figure \ref{fig:H}.
Notice the multiple edges at each level.

\begin{figure}
\begin{center}
	\includegraphics
{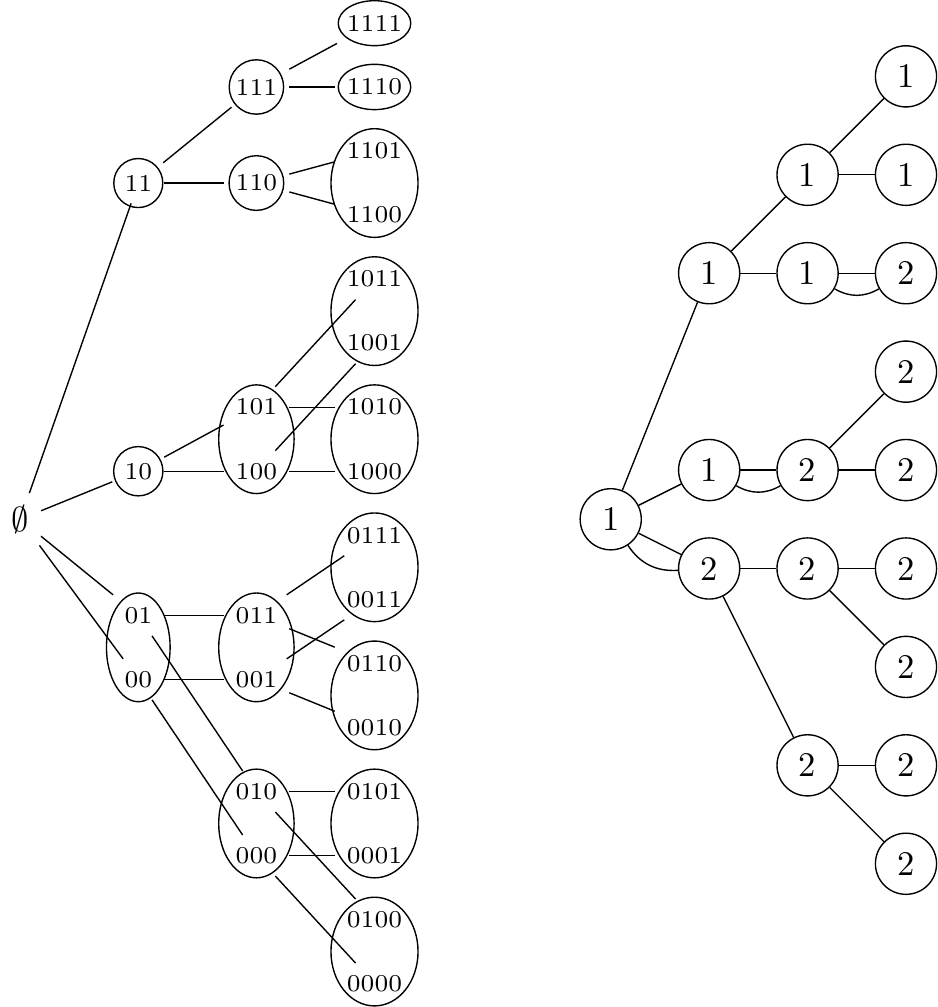}
\end{center}
	\caption{First levels of Bratteli diagrams of $H$}
	\label{fig:H}	
\end{figure}

On the other hand, the figures give us a clue as to what precisely is the difference between $G$ and $H$. 
Identifying the two top-most all-1s infinite paths in the Bratteli diagram in Figure \ref{fig:G},  we obtain a graph isomorphic to that in Figure \ref{fig:H}. 

Indeed, the action of $H$ is a quotient of that of $G$.
Let $X'$ be the quotient $X/\{01^\infty \sim 1^\infty\}$ obtained by identifying the fixed points of $G\leq\Homeo(X)$.
The action of $G$ on $X$ can naturally be pushed forward to an action of $G$ on $X'$; explicitly, the image of this action is 
$\Full(\grp{h'_m \mid m \ge 0})$ where $h'_m\in \Homeo(X')$ acts by swapping $01^m0w$ and $11^m0w$ for $w \in \{0,1\}^{\N}$.
This action is topologically conjugate to that of $H$ via the homeomorphism $\theta\colon X \rightarrow X'$ defined  by 
 $$\theta(1^\infty) = \{01^\infty,1^\infty\}, \qquad \theta(1^m0\delta w) = \delta 1^m0w,  \text{ for all } m \ge 0, \; \delta \in \{0,1\}, \; w\in\{0,1\}^{\N}.$$
   Indeed, $$\theta(h_m(1^m0\delta w))=\theta(1^m0\overline{\delta}w)= \overline{\delta}1^m0w=h'_m(\delta1^m0w)=h'_m(\theta(1^m0\delta w))$$ for all $m\geq 0$ and $w\in\{0,1\}^{\N}$ where $\overline{\delta}$ denotes the opposite value of $\delta=0,1$.
 Therefore 
 $\theta$ is a conjugation from  $(X,\Full(\grp{h_m}))$ to $(X',\Full(\grp{h'_m}))$.

This illustrates how one can obtain an action that is not uniformly discrete as a quotient of a uniformly discrete action.

\end{document}